\documentclass{amsart}
\usepackage{amsmath,amsfonts,amsthm,amssymb,amscd,stmaryrd,url,amstext, amsxtra,amsopn}
\usepackage{mathtools}
\usepackage{xcolor}

\usepackage{hyperref}

\DeclareFontFamily{OT1}{rsfs}{}
\DeclareFontShape{OT1}{rsfs}{n}{it}{<-> rsfs10}{}
\DeclareMathAlphabet{\mathscr}{OT1}{rsfs}{n}{it}

 \DeclareMathOperator{\Li}{Li}
 \DeclareMathOperator{\Gal}{Gal}
 \DeclareMathOperator{\GL}{GL}

  \DeclareMathOperator{\cd}{cd}
  \DeclareMathOperator{\Irr}{Irr}

 \newcommand{\mymod}[1]{(\operatorname{mod} #1)}

\def\q{\mathfrak{q}}
\def\p{\mathfrak{p}}
\def\a{\mathfrak{a}}

\def\m{\mathfrak{m}}

\def\e{\mathfrak{e}}

\def\norm{{\rm N}}       

\makeatletter
\let\@@pmod\pmod
\DeclareRobustCommand{\pmod}{\@ifstar\@pmods\@@pmod}
\def\@pmods#1{\mkern4mu({\operator@font mod}\mkern 6mu#1)}
\makeatother

\newtheorem{theorem}{Theorem}
 
 \newtheorem{lemma}[theorem]{Lemma}

   \theoremstyle{remark}
 \newtheorem*{remark}{Remark}
 
\reversemarginpar 

\begin{document}

\title[Cyclicity and Exponent of Elliptic Curves Modulo $p$ in AP]{Cyclicity and Exponent of  Elliptic Curves\\Modulo $p$ in Arithmetic Progressions}

   \dedicatory{Dedicated to Professor Ram Murty on the occasion of his seventieth birthday}

\author{Peng-Jie Wong}

\address{Department of Applied Mathematics\\
National Sun Yat-Sen University\\
Kaohsiung City, Taiwan}
\email{pjwong@math.nsysu.edu.tw}
\subjclass[2010]{11G05, 11N36, 11G15, 11R45} 

\thanks{This research was partially supported by the NSTC grant 111-2115-M-110-005-MY3.}


\keywords{Elliptic curves, cyclicity problem, exponents of elliptic curves, primes in arithmetic progressions}
\date{}

\begin{abstract}
In this article, we study the cyclicity problem of elliptic curves $E/\Bbb{Q}$ modulo primes in a given arithmetic progression. We extend the recent work of Akbal and G\"ulo\u{g}lu by proving an unconditional asymptotic for such a cyclicity problem over arithmetic progressions for CM elliptic curves $E$, which also presents a generalisation of the previous works of Akbary, Cojocaru, M.R. Murty, V.K. Murty, and Serre. In addition, we refine the conditional estimates of Akbal and G\"ulo\u{g}lu, which gives log-power savings (for small moduli) and consequently improves the work of Cojocaru and M.R. Murty. Moreover, we study the average exponent of $E$ modulo primes in a given arithmetic progression and obtain several conditional and unconditional estimates, extending the previous works of Freiberg, Kim, Kurlberg, and Wu.
\end{abstract}

\maketitle

\section{Introduction}

Let $E$ be an elliptic curve, defined over $\Bbb{Q}$, of conductor $N_E$. For a prime $p$ of good reduction, we let $\bar{E}$ denote the reduction of $E$ modulo $p$ and $\bar{E}(\Bbb{F}_p)$ be the group of rational points of $\bar{E}$ over $\Bbb{F}_p$. The study of $\bar{E}(\Bbb{F}_p)$, as $p$ varies, manifests as part of the ``analytic theory'' of elliptic curves. Most profoundly, Lang and Trotter \cite{LT} formulated an elliptic curve analogue of Artin's primitive root conjecture, and Serre \cite{Se85} considered the cyclicity problem of estimating
$$
\pi_{c}(x,E)=\#\{p \le x \mid \text{$p\nmid N_E$ and $\bar{E}(\Bbb{F}_p)$ is cyclic}\}. 
$$
For each $m\in\Bbb{N}$, let $E[m]$ denote the group of $m$-torsion points of $E$, and let $\Bbb{Q}(E[m])$ be the $m$-th division field of $E$. In light of Hooley's conditional resolution of Artin's primitive root conjecture, Serre proved that assuming the generalised Riemann hypothesis (GRH), if $\Bbb{Q}(E[2])\neq \Bbb{Q}$, then one has
$$
\pi_{c}(x,E) = \mathfrak{c}_E \Li(x) + o\left( \frac{x}{\log x}\right),
$$
where $\Li(x)$ is the usual logarithmic integral,
 $$
\mathfrak{c}_E=\sum_{m=1}^{\infty}\frac{\mu(m)}{[\Bbb{Q}(E[m]):\Bbb{Q}]},
$$
and $\mu(m)$ is the M\"obius function. This was improved by  M.R. Murty \cite[pp. 160-161]{MR83}, who showed that under GRH, one has
$$
\pi_{c}(x,E) = \mathfrak{c}_E \Li(x) + O\left( \frac{x(\log\log x)}{(\log x)^2}\right).
$$ 
Moreover, the error term above was sharpened considerably by Cojocaru and M.R. Murty \cite{CM04}. They proved under GRH that for any elliptic curve $E/\Bbb{Q}$  of conductor $N_E$, if $E$ is with complex multiplication (CM) by the full ring of integers $\mathcal{O}_K$ of an imaginary quadratic field $K$, one has
\begin{equation}\label{CMe}
\pi_{c}(x,E) = \mathfrak{c}_E \Li(x) + O( x^{3/4}(\log (N_E x))^{{1}/{2}}),
\end{equation}
and if $E$ is non-CM, one has
\begin{equation}\label{CoMu-non-CM}
\pi_{c}(x,E) = \mathfrak{c}_E \Li(x) + O(x^{5/6} (\log (N_E x))^{2/3})  + O\left( \frac{(\log\log x) (\log (N_E x) )}{\log x}  A(E)^3 \right), 
\end{equation}
where $A(E)$ is the Serre's constant associated to $E$.\footnote{Recall that for each $m$, there is a natural injective representation $\rho_m:\Gal(\Bbb{Q}(E[m])/\Bbb{Q})\rightarrow \GL_2(\Bbb{Z}/m\Bbb{Z})$ associated to $E$. Serre \cite{Se71} proved that if $E$ is non-CM, then there exists a finite set $S_E$ of primes such that $\rho_\ell$ is surjective whenever $\ell\notin S_E$. Furthermore, setting
\begin{equation}\label{def-A(E)}
A(E)=2\cdot 3\cdot 5 \cdot \prod_{\ell\in S_E \backslash\{2,3,5\}} \ell,
\end{equation}
Serre's constant associated to $E$, $\rho_m$ is surjective when $(m,A(E))=1$ (see \cite[Appendix]{Co04}).}

There are several unconditional results regarding $\pi_{c}(x,E)$. For any elliptic curve $E/\Bbb{Q}$, Gupta and M.R. Murty \cite{GM90} showed that $\pi_{c}(x,E)\gg x/(\log x)^2$. Moreover, for CM elliptic curves,  the assumption of GRH was removed by M.R. Murty in \cite{MR83} (where Wilson's Bombieri-Vinogradov theorem for number fields  \cite{Wil} was used, and no error terms were given). Furthermore, by the sieve of Eratosthenes and the effective version of the Chebotarev density theorem due to  Lagarias and  Odlyzko \cite{LO77} (instead of using  Wilson's theorem), Cojocaru \cite{Co03} proved that for any CM elliptic curve $E/\Bbb{Q}$ of conductor $N_E$,
$$
\pi_{c}(x,E) = \mathfrak{c}_E \Li(x) + O\left( \frac{x}{(\log x)(\log\log ((\log x)/N_E^2))}\frac{\log\log x}{\log ((\log x)/N_E^2)}\right).
$$
Moreover, adapting M.R. Murty's argument and applying the Bombieri-Vinogradov theorem for number fields  established by Huxley \cite{Hu86}, Akbary and  V.K. Murty \cite{AM} obtained the improvement that given any CM elliptic curve  $E/\Bbb{Q}$  of conductor $N_E$, for any $A, B> 0$, one has 
\begin{align}\label{AMe}
\pi_{c}(x,E) =\mathfrak{c}_E\Li(x)+O_{A,B}\left(\frac{x}{(\log x)^A}\right)
\end{align}
uniformly in $N_E \le (\log x)^B$, where the implied constant only depends on $A$ and $B$.

Recently, Akbal and G\"ulo\u{g}lu \cite{AG} proposed the cyclicity problem over arithmetic progressions, asking for estimates of
$$
\pi_{c}(x,E,q,a) = \#\{p \le x \mid \text{$p\nmid N_E$, $p\equiv a\,\mymod{q}$, and $\bar{E}(\Bbb{F}_p)$ is cyclic}\}
$$
when $(a,q)=1$. By extending the argument of Gupta and M.R. Murty \cite{GM90}, they proved the following theorem.

\begin{theorem}[{\cite[Theorem 1]{AG}}] \label{AG-uncond}
 Let $E/\Bbb{Q}$ be an elliptic curve, and let $a$ and $q$ be coprime natural numbers such that $(a-1, q )$ has no odd prime divisors. Assume that $[\Bbb{Q}(E[2]):\Bbb{Q}] = 3$. Then for any fixed $A\ge 0$ and sufficiently large $x$, when $q \ll (\log x)^A$,  one has $\pi_{c}(x,E,q,a)
 \gg x/(\log x)^{2+A}$ unless $\Bbb{Q}(E[2]) \subseteq \Bbb{Q}(\zeta_q)$ and $\sigma_a$ fixes $\Bbb{Q}(E[2])$. (Here, as later, $\Bbb{Q}(\zeta_q)$ is the $q$-th cyclotomic extension of the rationals formed by adjoining a primitive $q$-th root of unity $\zeta_q$, and $\sigma_a$ denotes the automorphism $\zeta_q \mapsto \zeta_q^a$.)
\end{theorem}

Furthermore,  under GRH,  by generalising the work of Cojocaru and M.R. Murty \cite{CM04}, Akbal and G\"ulo\u{g}lu \cite{AG} determined the asymptotics for $\pi_{c}(x,E,q,a)$ as follows.

\begin{theorem}[{\cite[Theorems 3 and 5]{AG}}] \label{AG-main-thm}
 Let $E/\Bbb{Q}$ be an elliptic curve of conductor $N_E$, and let $a$ and $q$ be coprime natural numbers. Assume that GRH is valid for the Dedekind zeta function of  $\Bbb{Q}(E[m])\Bbb{Q}(\zeta_q)$ for every square-free $m$. Define
$$
 \mathfrak{c}_E(q,a) =\sum_{m=1}^{\infty}\frac{ \gamma_{E,m}(q,a)\mu(m)}{[\Bbb{Q}(E[m])\Bbb{Q}(\zeta_q):\Bbb{Q}]}, 
$$
where $\gamma_{E,m}(q,a) = 1$ if the automorphism $\sigma_a$, defined as in Theorem \ref{AG-uncond},  fixes $\Bbb{Q}(E[m]) \cap \Bbb{Q}(\zeta_q )$, and it is 0 otherwise.
 Then one has
$$
\pi_{c}(x,E,q,a)  = \mathfrak{c}_E(q,a) \Li(x) +\mathcal{E}_c(x),
$$  
where  if $E$ is with complex multiplication by the full ring of integers $\mathcal{O}_K$ of an imaginary quadratic field $K =\Bbb{Q}(\sqrt{-D})$, $\mathcal{E}_c(x)$ satisfies
\begin{align*}
 \begin{split}
\mathcal{E}_c(x) 
&\ll  x^{3/4} \left(\frac{\log(qN_Ex) G_D(a,q)}{q^3} \right)^{1/2} 
+ x^{3/4} \left(\frac{\log(qN_Ex)}{\log x} \right)^{1/2}\\
&+x^{1/2} q \log(qN_E x) + x^{1/2} \left( \frac{1}{q} +\frac{\log x}{q^2} \right)  G_D(a,q)
  \end{split}
\end{align*}
with $G_D(a,q)<  c\cdot 4^{\omega(q)} \tau_2(q) q^2,$ (here, $c = 2$ if $D\equiv 1,2\enspace\mymod{4}$,  or $D\equiv 3\enspace\mymod{4}$ and $q$ is odd; $c = 49$ otherwise); and if $E$ is non-CM, one has
\begin{align}\label{AG-non-CM-GRH}
 \begin{split}
\mathcal{E}_c(x)&\ll  x^{5/6} \left( \frac{H(q) (\log  ( qN_Ex))^2}{q}\right)^{1/3}
+x^{5/8}  \left( \frac{\tau_2(q_2)(\log  ( qN_Ex))^3}{\phi(q)\log x } M_E^3 \right)^{1/4} \\
& + x^{1/2} q\log  ( qN_Ex) + \frac{\tau_2(q_2)}{\phi(q) x^{1/2}\log x }M_E^3,
  \end{split}
\end{align}
where $q_2$ denotes the largest divisor of $q$ that is coprime to $M_E$ defined in \eqref{def-ME} below.
\end{theorem}

Here, as later, $\omega(n)$ denotes the number of prime divisors of $n$, and $\tau_2(n)$ is the number of (positive) divisors of $n$. In addition, $\phi(n)$ is Euler's totient function, and the arithmetic function $H(n)$ is defined by
$$
H(n)= \sum_{d\mid n} \sum_{\substack{1\le k\le d\\ d\mid k^2 }} 1.
$$
Also, $G_D(a,q)$ is the cardinality of the set defined in \cite[Eq. (23)]{AG}, and its bound is given in \cite[Eq. (11)]{AG}. Last but not least, $M_E\in\Bbb{N}$ is defined by
\begin{equation}\label{def-ME}
M_E =\prod_{\ell\mid A(E)N_E} \ell,
\end{equation}
where $A(E)$ is the Serre's constant defined in \eqref{def-A(E)}.

One of the main objectives of this article is to prove the following unconditional asymptotic of $\pi_{c}(x,E,q,a)$ for CM elliptic curves $E$.

\begin{theorem}\label{main-thm-c}
Let $E/\Bbb{Q}$ be a CM elliptic curve of conductor $N_E$, and let $a$ and $q$ be coprime natural numbers. Then for any $A, B> 0$, we have
\begin{align}\label{AMe-generalisation}
\pi_{c}(x,E,q,a) = \mathfrak{c}_E(q,a) \Li(x)+O_{A,B}\left(\frac{x}{(\log x)^A}\right),
\end{align}
uniformly in $qN_E \le (\log x)^B$, where the implied constant only depends on $A$ and $B$.
\end{theorem}

In addition,  we have the following refinement of Theorem \ref{AG-main-thm}, which notably improves the estimates \eqref{CMe} and \eqref{CoMu-non-CM} of Cojocaru and M.R. Murty by factors of $(\log x)^{1/2}$ and $(\log x)^{1/3}$, respectively.

\begin{theorem}\label{main-thm-c-GRH}
Let $E/\Bbb{Q}$ be an elliptic curve of conductor $N_E$, and let $a$ and $q$ be coprime natural numbers. Assume GRH. If $E$ has CM, then we have
\begin{equation}\label{pic-CM-GRH}
\pi_{c}(x,E,q,a) = \mathfrak{c}_E(q,a)\Li(x)+ O\left( x^{3/4} \frac{( \log  ( qN_Ex))^{1/2}}{(\log x)^{1/2}}
  + x^{1/4} \log N_E \right).
\end{equation}
In particular, we have
$$
\pi_{c}(x,E) =  \mathfrak{c}_E \Li(x) +
O\left( x^{3/4} \frac{( \log  ( N_Ex))^{1/2}}{(\log x)^{1/2}}
  + x^{1/4} \log N_E \right).
$$
Furthermore, if $E$ is non-CM, we have
\begin{align}\label{pic-aq-GRH}
 \begin{split}
&\pi_{c}(x,E,q,a)\\
& = \mathfrak{c}_E(q,a) \Li(x)+  O\left( x^{5/6} \frac{ (\log (qN_E x))^{2/3} }{ (\log x)^{1/3}}  +  \frac{\tau_2(q_2)  \log(qN_E x)}{  \phi(q) }R_{E,q_1}  \right),
\end{split}
\end{align}
where $q_1 = \frac{q}{q_2}$,  $q_2$ denotes the largest divisor of $q$ that is coprime to $M_E$, and
\begin{equation}\label{def-REq}
R_{E,q_1} =   \sum_{d\mid M_E }  \frac{\phi((d,q_1)) d^3}{ \phi(d)  }.
\end{equation}
Consequently, we have
$$
\pi_{c}(x,E) =  \mathfrak{c}_E \Li(x) + O\left(x^{5/6} \frac{ (\log (N_E x))^{2/3} }{ (\log x)^{1/3}}  +   \log(N_E  x)    \sum_{d\mid M_E } \frac{ d^3}{ \phi(d)  }\right).
$$
In addition, the both factors $M_E^3$ in \eqref{AG-non-CM-GRH} can be replaced by  $R_{E,q_1}$.
\end{theorem}

To obtain the improved asymptotic \eqref{pic-CM-GRH}, we invoke the work of Hinz and Lodemann \cite{HL93} on the Brun-Titchmarsh inequality for number fields, which allows us to deduce a refined bound \eqref{CM-middle-refined} (cf. \eqref{tri_bd_Em} and \eqref{sum-y-piEm}). This is our key new observation. For the non-CM case, our results rely on an alternative bound for the ``middle range'' $\Sigma'_2$ in \eqref{Sigma2} and a refined bound for the ``tail'' $\Sigma_3$ in \eqref{Sigma3} arising from the sieving argument of Akbal and G\"ulo\u{g}lu \cite[Sec. 3.2]{AG}. When $H(q)/q$ is of a constant size, our result reduces the leading term of \eqref{AG-non-CM-GRH} by a factor of $(\log x)^{1/3}$.\footnote{It shall be noted that in general, the leading error term in \eqref{pic-aq-GRH} is smaller than the one in \eqref{AG-non-CM-GRH} only if $\frac{q}{H(q)} \ll (\log x)^{1/3}$. So, our result does \emph{not} yield a consistent improvement but relies on $q$. It is worth further noting that the reason why we were not able to give a uniform description for the range of $q$, where the improvement is granted, mainly came from the irregular behaviour of $H(q)$ as discussed in \cite[Remark 1]{AG}.} Such a ``log-saving'' is due to the insert of the Burn-Titchmarsh theorem when estimating $\Sigma'_2$. Also, when $q$ and  $M_E$ do not have too many common divisors, $R_{E,q_1}$ gives a better estimate than $M_E^3$. For instance, if $(q,M_E)=1$, $R_{E,q_1}=  \sum_{d\mid M_E } d^3/ \phi(d) $ presents a power-saving for the factors $M_E^3$ in \eqref{AG-non-CM-GRH}.

\begin{remark} 
\noindent (i) Note that  $
\pi_{c}(x,E,q,a)\le \pi(x,q,a) =  \#\{p\le x\mid  p\equiv a\,\mymod{q}\} $ and that under GRH for Dirichlet $L$-functions, for $(a,q)=1$, one has
$$
 \pi(x,q,a) = \frac{1}{\phi(q)}\Li(x) + O(x^{1/2} \log (qx)).
$$
Hence, when $\sqrt{x}\le q \le x $, assuming GRH, we have $\pi_{c}(x,E,q,a) \ll x^{1/2} \log x$, which provides a superior estimate than \eqref{pic-CM-GRH} and \eqref{pic-aq-GRH}.

\noindent (ii)
As may be noticed, applying Theorems \ref{AG-uncond} and \ref{main-thm-c}, for coprime $a,q\in \Bbb{N}$ such that $(a-1, q )$ has no odd prime divisors, if $E/\Bbb{Q}$ has CM and satisfies $[\Bbb{Q}(E[2]):\Bbb{Q}] = 3$, then $\mathfrak{c}_E(q,a)$ is positive unless $\Bbb{Q}(E[2]) \subseteq \Bbb{Q}(\zeta_q)$ and $\sigma_a$ fixes $\Bbb{Q}(E[2])$.\footnote{Under these assumptions, suppose, on the contrary, that $\mathfrak{c}_E(q,a)=0$. Then Theorem \ref{main-thm-c} gives  $\pi_{c}(x,E,q,a)  \ll   x/(\log x)^3$, which contradicts the estimate $\pi_{c}(x,E,q,a)\gg x/(\log x)^{2+\varepsilon}$ given by  Theorem \ref{AG-uncond} (with $A=\varepsilon\in(0,1)$).} In general, determining necessary and sufficient conditions for the positivity of $\mathfrak{c}_E(q,a)$ appears as an interesting question. For example, there is an observation of Serre that $\mathfrak{c}_E \neq 0 $ if and only if $E$ has an irrational 2-torsion point (see \cite[p. 619]{CM04} for a proof). Also, in \cite[Theorems 4 and 6]{AG}, Akbal and G\"ulo\u{g}lu gave some sufficient conditions for the positivity of $\mathfrak{c}_E(q,a)$. More recently, Jones and  Lee \cite{JL} systemically studied the question of which arithmetic progressions $a\,\mymod{q}$ admit the property that for all but finitely many primes $p \equiv a\,\mymod{q}$, $\bar{E}(\Bbb{F}_p)$ is not cyclic. In particular, they gave a criterion for $\mathfrak{c}_E(q,a)=0$ in \cite[Sec. 3.2]{JL}.
\end{remark}

In a slightly different vein, knowing that for any prime $p$ of good reduction, there are natural numbers $d_p$ and $e_p$ such that  $d_p\mid e_p$ and
$$
\bar{E}(\Bbb{F}_p)\simeq \Bbb{Z}/d_p\Bbb{Z} \oplus \Bbb{Z}/e_p\Bbb{Z},
$$
one may also study the behaviours of  $d_p$ and $e_p$ as $p$ varies.\footnote{The exponent  $e_p$ is the largest possible order of points on $\bar{E}(\Bbb{F}_p)$.} Indeed, determining the asymptotic for
$$
\pi_{e}(x,E) =  \sum_{p\le x}e_p
$$
presents an interesting problem.
Freiberg and Kurlberg \cite{FK14} investigated this problem and showed that under GRH, 
$$
\sum_{p\le x}e_p =\e_E \Li(x^2)+ O(x^{19/10}(\log x)^{6/5}), 
$$
where
$$
\e_E=\sum_{m=1}^{\infty}\frac{1}{[\Bbb{Q}(E[m]):\Bbb{Q}]}\sum_{de\mid m}\frac{\mu(d)}{e};
$$
they also showed that
$$
\sum_{p\le x}e_p =\e_E \Li(x^2)+ O\left( \frac{x^2 (\log \log \log x)}{(\log x)( \log \log x)}\right),
$$
unconditionally, if $E/\Bbb{Q}$ is a CM elliptic curve. The errors in these two estimates were improved by Wu \cite{Wu14} to $O(x^{11/6}(\log x)^{1/3})$ and  $O(x^{2}/(\log x)^{15/14})$, respectively. Furthermore, adapting the work of  Akbary and V.K. Murty  \cite{AM}, Kim \cite{K14} used Huxley's Bombieri-Vinogradov theorem for number fields \cite{Hu86} to show that if $E/\Bbb{Q}$ has CM, then for any $A, B> 0$, one has
\begin{align}\label{K14}
\pi_{e}(x,E) =\mathfrak{e}_E\Li(x^2)+O_{A,B}\left(\frac{x^2}{(\log x)^A}\right)
\end{align}
uniformly in $N_E \le (\log x)^B$, where the implied constant only depends on $A$ and $B$.

Inspired by the previously-mentioned work of Akbal and G\"ulo\u{g}lu \cite{AG}, we consider the average of exponents over arithmetic progressions:
$$
\pi_{e}(x,E,q,a) = \sum_{\substack{p\le x\\ p\equiv a\mymod{q} }} e_p
$$
and prove the following generalisation of Kim's work \cite{K14}. (Note that when $q=1$, our result recovers Kim's estimate \eqref{K14}.)

\begin{theorem}\label{main-thm-exp}
Let $E/\Bbb{Q}$ be a CM elliptic curve of conductor $N_E$, and let $a$ and $q$ be coprime natural numbers. Then setting
$$
 \mathfrak{e}_E(q,a) =\sum_{m=1}^{\infty}\sum_{de\mid m}\frac{\mu(d)}{e} 
  \frac{ \gamma_{E,m}(q,a)\mu(m)}{[\Bbb{Q}(E[m])\Bbb{Q}(\zeta_q):\Bbb{Q}]} ,
$$
for any $A, B> 0$, we have
\begin{align}\label{SW-pie}
\pi_{e}(x,E,q,a) = \mathfrak{e}_E(q,a) \Li(x^2)+O_{A,B}\left(\frac{x^2}{(\log x)^A}\right)
\end{align}
uniformly in $qN_E\le (\log x)^B$, where the implied constant only depends on $A$ and $B$.
\end{theorem}


Moreover, we have the following extension of the works of Freiberg-Kurlberg and Wu mentioned earlier, which particularly presents refinements of their results by taking $q = 1$.

\begin{theorem} \label{main-thm-exp-cond}
 Let $E/\Bbb{Q}$ be an elliptic curve of conductor $N_E$, and let $a$ and $q$ be coprime natural numbers. Assume that GRH is valid for the Dedekind zeta function of  $\Bbb{Q}(E[m])\Bbb{Q}(\zeta_q)$ for every square-free $m$. 
 Then 
$$
\pi_{e}(x,E,q,a)  = \mathfrak{e}_E(q,a) \Li(x^2) +x\mathcal{E}_e(x),
$$  
where  if $E$ is with CM by the full ring of integers $\mathcal{O}_K$ of an imaginary quadratic field $K =\Bbb{Q}(\sqrt{-D})$, $\mathcal{E}_e(x)$ satisfies both bounds
\begin{equation*} 
\mathcal{E}_e(x)  \ll x^{3/4} \frac{( \log  ( qN_Ex))^{1/2}}{(\log x)^{1/2}}
  + x^{1/4} \log N_E 
\end{equation*}
and
\begin{align}\label{pie-CM-GRH}
 \begin{split}
\mathcal{E}_e(x) 
&\ll x^{3/4} \left(\frac{\log(qN_Ex) G_D(a,q)}{q^3} \right)^{1/2}
+x^{3/4} \left(\frac{\log(qN_Ex)}{\log x} \right)^{1/2}\\
&+x^{1/2} q \log(qN_Ex) + x^{1/2} \left( \frac{1}{q} +\frac{\log x}{q^2} \right)  G_D(a,q), 
  \end{split}
\end{align}
with the same $G_D(a,q)$ as in Theorem \ref{AG-main-thm}; and if $E$ is non-CM, $\mathcal{E}_e(x)$ satisfies both estimates
\begin{equation}\label{pie-nonCM-GRH-1}
\mathcal{E}_e(x)\ll x^{5/6} \frac{ (\log (qN_E x))^{2/3} }{ (\log x)^{1/3}}  +  \frac{x^{1/2}}{q} 
\end{equation}
and
\begin{align}\label{pie-nonCM-GRH-2}
 \begin{split}
\mathcal{E}_e(x)&\ll  x^{5/6} \left( \frac{H(q) (\log  ( qN_Ex))^2}{q}\right)^{1/3} +x^{5/8}  \left( \frac{\tau_2(q_2)(\log  ( qN_Ex))^3}{\phi(q_2)\log x }S_{E}\right)^{1/4} \\
&+x^{1/2} q\log  ( qN_Ex) + \frac{\tau_2(q_2)}{\phi(q_2) x^{1/2}\log x }S_{E},
  \end{split}
\end{align}
where $S_{E} = \sum_{d\mid M_E^\infty }  \frac{B_E}{d \phi(d) } $, and $B_E$ is the constant, depending only on $E$, as in \cite[Proposition  3.2]{FK14}, such that $ B_E \cdot [\Bbb{Q}(E[m]):\Bbb{Q}]\ge |\GL_2(\Bbb{Z}/m\Bbb{Z})|$ for every $m$.\footnote{The existence of $ B_E$ is due to  Serre's open image theorem. It may be possible to determine $ B_E$ in terms of $N_E$ by an effective version of Serre's open image theorem (see, e.g., \cite{MW} and references therein) or  the index bound for the image of the ``adelic'' representation attached to $E$ (see \cite{Lo15}). However, as it seemingly requires a delicate algebraic and representation-theoretic argument that appears beyond this article's scope, we shall reserve it as a future project.}
\end{theorem}

\begin{remark} 
(i) It is worthwhile noting that our unconditional estimates \eqref{AMe-generalisation} and \eqref{SW-pie} come from the effective version of the Chebotarev density theorem established by V.K. Murty \cite{VKMu97}. This is the main observation in the proofs of Theorems \ref{main-thm-c} and \ref{main-thm-exp} that verifying Artin's (holomorphy) conjecture for the Galois extensions $L/\Bbb{Q}$ involved (see Lemma \ref{AC}) allows us to remove $n_L$, the degree of $L$, in the error term of the effective version of the Chebotarev density theorem \eqref{LO-CDT}  due to Lagarias and  Odlyzko \cite{LO77}. 

\noindent (ii) Compared to the works \cite{AM,K14}, our argument does not rely on Huxley's Bombieri-Vinogradov theorem for number fields. Still, it obtains results of the same strength (which also particularly gives an improvement of Cojocaru's work \cite{Co03}). Moreover, our method allows us to express the errors in terms of the location of the possible Landau-Siegel zeros of Dirichlet $L$-functions. This feature could not be seen from the method relying on the Bombieri-Vinogradov theorem for number fields, and it leads to a conditional resolution of a question of Akbary and V.K. Murty on improving the error term in their estimate \eqref{AMe} as discussed below.
\end{remark}

Akbary and V.K. Murty \cite{AM} remarked that their theorem can be viewed as an elliptic analogue of the following weak form of the classical Siegel-Walfisz theorem: for any $(a,q)=1$, one has
$$
\pi(x,q,a)
 = \frac{1}{\phi(q)} \Li(x) +  O_{A,B}\left(\frac{x}{(\log x)^A} \right),
$$
uniformly for $q \le (\log x)^B$, for any given $A,B>0$. Moreover, recalling that the Siegel-Walfisz theorem, in fact, states that for any $B > 0$, there exists $c_0 = c_{0,B}$ such that 
$$
\pi(x,q,a)
 = \frac{1}{\phi(q)} \Li(x) +   O\left(  x\exp( -c_0 \sqrt{\log x}  ) \right)
$$
uniformly in $q \le (\log x)^B$, Akbary and V.K. Murty noted that it seems quite unclear how to extend this to their setting (i.e., to the cyclicity problem). In Section \ref{final-rmk}, we shall show that such an expected stronger estimate follows from the non-existence of the Landau-Siegel zeros of Dirichlet $L$-functions. Indeed, we have the following conditional result.

 \begin{theorem}\label{main-thm-NLS}
Let $E/\Bbb{Q}$ be a CM elliptic curve of conductor $N_E$, and let $a$ and $q$ be coprime natural numbers. Assume that there exists a constant $S\ge -1$ such that for any $Q\in\Bbb{N}$ and any real primitive character $\chi$ modulo $Q$,
\begin{equation}\label{lower-bd-L}
L(1,\chi) \gg (\log Q)^{-S},
\end{equation}
where $L(s,\chi)$ is the Dirichlet $L$-function attached to $\chi$, and the implied constant is absolute. Then there is an absolute constant $c_1>0$ so that uniformly for  $ \log (qN_E )  \ll (\log x)^{1/(2S+4)}$,\footnote{In fact, such uniformity can be precisely written as
$qN_E   \le \exp\left( \frac{1}{2\kappa} (\log x)^{1/(2S+4)}\right)$ with the same $\kappa$ as in \eqref{E-logQ-bd}.} we have
\begin{align*}
\pi_c(x,E,q, a)
= \mathfrak{c}_E(q,a) \Li(x) 
+ O\left(  x \exp \left(- c_1 (\log x)^{1/(2S+4)}   \right)\right)
\end{align*}
and
\begin{align*}
  \begin{split}
\pi_{e}(x,E,q,a)
=  \mathfrak{e}_E(q,a)\Li(x^2)+ O\left(   x^2 \exp \left( -c_1 (\log x)^{1/(2S+4)}   \right)\right).
 \end{split}
\end{align*}
\end{theorem}

\begin{remark} 
It is well-known that the non-existence of the Landau-Siegel zeros implies that \eqref{lower-bd-L} holds with $S=-1$ (see \cite[\textsection 1]{Go75}). In his recent preprint \cite{Zh22}, Zhang announced that $S = 2022$ is admissible; however, as \cite{Zh22} is still unpublished,  Theorem \ref{main-thm-NLS} shall be treated with caution as a conditional result.
\end{remark}

The rest of the article is organised as follows. In the next section, we will collect the necessary preliminaries to prove our results (particularly, we will discuss the effective version of the Chebotarev density theorem established by V.K. Murty). Theorems \ref{main-thm-c}, \ref{main-thm-c-GRH}, \ref{main-thm-exp}, and \ref{main-thm-exp-cond} will be proved in Sections \ref{proof},  \ref{proof-pic-GRH}, \ref{proof-2}, and \ref{proof-pie-GRH}, respectively. In the last section, we will prove Theorem \ref{main-thm-NLS}.


\section{Preliminaries}

\subsection{Artin's (holomorphy) conjecture and the Chebotarev density theorem}

In this section, we shall recall the effective version of the Chebotarev density theorem established by V.K. Murty. To state his result, we require the following notation. For a number field $F$, we let $d_F$ and $n_F$ denote the absolute discriminant and degree of $F$, respectively. Let $L/K$ be a Galois extension of number fields with Galois group $G$. The set of irreducible characters of $G$ will be denoted by $\Irr(G)$, and the biggest character degree of $G$ is defined by $b(G)=\max_{\chi\in\Irr(G)} \chi(1)$. Also, for each $\chi \in\Irr(G)$, we let $\mathfrak{f}_{\chi}$ stand for the (global) Artin conductor of $\chi$, and we set
$$
\mathcal{A}= \mathcal{A}(L/K) =\max_{\chi\in\Irr(G)}A_{\chi},
$$
where $A_\chi = d_K^{\chi(1)} \norm(\mathfrak{f}_{\chi})$. Moreover, we recall from \cite[Proposition 2.5]{MMS88} that
\begin{equation}\label{cond-bd}
\log \norm(\mathfrak{f}_{\chi}) \le 2\chi(1)n_K \left(\sum_{p\in P(L/K)} \log p  +  \log (n_{L}/n_K)\right),
\end{equation}
where $P(L/K)$ denotes the set of rational primes $p$ for which there is a prime $\mathfrak{p}\mid p$ of $K$ such that $\mathfrak{p}$ is ramified in $L$. Following \cite[Sec. 4]{VKMu97}, we further put
\begin{equation}\label{def-M}
\log \mathcal{M} = \log \mathcal{M}(L/K)
= \frac{1}{n_K} \log d_K + 2\sum_{p\in P(L/K)} \log p  +2 \log (n_{L}/n_K).
\end{equation}

Let $C$ be a conjugacy class of $G=\Gal(L/K)$, and let $\pi_C(x)$ denote the number of primes $\mathfrak{p}$ of $K$, with $\norm(\mathfrak{p})\le x$, whose Artin symbol equals $C$. In \cite{VKMu97}, V.K. Murty proved the following effective version of the Chebotarev density theorem under Artin's conjecture.

\begin{theorem}[{\cite[Theorem 4.1]{VKMu97}}] \label{VKMurtythm}
If Artin's conjecture holds for $L/K$, then there is an absolute $c_2>0$ such that for any conjugacy class $C$ in $G$,
one has
\begin{align}\label{VKMurty-est}
 \begin{split}
\pi_C(x) &=\frac{|C|}{|G|} \Li(x) - \frac{|C|}{|G|}\chi_1(C) \Li(x^{\beta_1})\\
& +O\left(|C|^{\frac{1}{2}} n_K x (\log (\mathcal{M} x))^2
\exp\left(   \frac{-c_2 \log x}{b(G)^{\frac{3}{2}}  \sqrt{ b(G)^3 (\log\mathcal{A} )^2 + n_K\log x }}\right) \right)
  \end{split}
\end{align}
provided that $\log x\gg b(G)^4 n_K \log \mathcal{M}$. Here, $\beta_1$ is the possible exceptional zero of $\zeta_L(s)$, the Dedekind zeta function of $L$. If $\beta_1$ exists, then it must arise from $L(s,\chi_1,L/K)$, the Artin $L$-function attached to $\chi_1$, for some character $\chi_1$ that is real and abelian (i.e 1-dimensional).
\end{theorem}

When $K=\Bbb{Q}$, we know that $\chi_1$ corresponds to a real (quadratic) Dirichlet character modulo $Q_1$ (say). By \eqref{cond-bd}, we have
\begin{equation}\label{logQ-bd}
\log Q_1 \ll   \sum_{p\in P(L/\Bbb{Q})} \log p  + \log n_{L},
\end{equation}
which will allow us to apply the following theorem of Siegel later. 

\begin{theorem}[Siegel]\label{Siegel}
With the same notation as above, for any $\varepsilon>0$, there is an absolute (ineffective) constant $c_3(\varepsilon)>0$ such that
$$
\beta_1< 1-\frac{c_3(\varepsilon)}{Q_1^\varepsilon}.
$$
\end{theorem}

\subsection{Elliptic curves and associated prime-counting functions}\label{pre}

Let $E/\Bbb{Q}$ be an elliptic curve of conductor $N_E$, and let $E[m]$ denote the set of $m$-torsion points of $E$. It is well-known that  $\Bbb{Q}(E[m])/\Bbb{Q}$ forms a Galois extension. Also, the ramified primes of $\Bbb{Q}(E[m])/\Bbb{Q}$ are divisors of $mN_E$  (see \cite[Proposition 3.5]{CM04}). Consequently, as $\Bbb{Q}(\zeta_q)\subset \Bbb{Q}(E[q])$, we know that $\Bbb{Q}(E[m])\Bbb{Q}(\zeta_q)\subset \Bbb{Q}(E[mq]) $, and thus the ramified primes of $\Bbb{Q}(E[m])\Bbb{Q}(\zeta_q)$ must divide $mqN_E$.

We shall require the following handy lemmata  (see, e.g., \cite[Lemma 2.1]{CM04} and \cite[Lemma 2.2]{K14}).

\begin{lemma}\label{cyc-cri}
Assume that $p\nmid N_E$. Then $\bar{E}(\Bbb{F}_p)$ is cyclic if and only if $p$ does not split completely in $\Bbb{Q}(E[m])$ for any square-free $m>1$.
\end{lemma}

\begin{lemma}\label{cri_d}
Let $E/\Bbb{Q}$ be an elliptic curve and $p$ be a prime of good reduction. Then $p$ splits completely in $\Bbb{Q}(E[m])$  if and only  if  $m\mid d_p$.
\end{lemma}

We set $\pi_{E,1}(x,q,a)=\#\{p\le x\mid p\nmid N_E,\, p\equiv a\,\mymod{q}\}$, and for integers $m>1$, we define
$$
\pi_{E,m}(x,q,a)=\#\{p\le x\mid \text{$p\nmid N_E$ splits completely in $\Bbb{Q}(E[m])$, $p\equiv a\,\mymod{q}$}\}.
$$

In the remaining part of this section, we shall further assume that $E$ has CM by the ring of integers $\mathcal{O}_K$ of an imaginary quadratic field $K$. Recall that for any $3\le m\le \sqrt{x} +1$, one has
\begin{equation}\label{tri_bd_Em}
\pi_{E,m}(x,q,a) 
\le \#\{p\le x\mid \text{$p\nmid N_E$ splits completely in $\Bbb{Q}(E[m])$}\}
\ll \frac{x}{m^2},
\end{equation}
where the implied constant is absolute (see, e.g., \cite[Lemma 2.3]{K14} or \cite[Lemma 5]{MR83}). In addition, the argument used in  \cite[Sec. 4.1]{AG} yields 
\begin{equation}\label{sum-y-piEm}
\sum_{y< m \le \sqrt{x} +1} \pi_{E,m}(x,q,a)
\ll  \left( \frac{\sqrt{x}}{q} + \frac{\sqrt{x}\log x}{q^2} + \frac{x}{yq^3} \right) G_D(a,q),
\end{equation}
for $2q\le y\le\sqrt{x}$, where $G_D(a,q)$ is the same as in Theorem \ref{AG-main-thm}.



\begin{remark}
It may be noticed that in \eqref{tri_bd_Em}, the dependence of $q$ is dropped. It is undoubtedly desirable to obtain a uniform upper bound for $\pi_{E,m}(x,q,a) $ with explicit dependence of $q$ (whenever $E$ is with CM or not). In general, it manifests as a Brun-Titchmarsh type inequality for the Chebotarev density theorem. The known results usually require $x$ significantly greater than a power of the absolute discriminant of $\Bbb{Q}(E[m])$, (see, e.g., \cite[Theorem 1.4]{LMO}), which is often too large for our purpose. On the other hand, for CM curves, the proof of  \eqref{tri_bd_Em} translates the original estimate to a lattice-counting  problem over quadratic fields (see \cite[Lemma 5]{MR83}). However, tracing the involving mod $q$ condition appears to be unclear during the translating procedure. To a certain degree, this reflects the fact that the degrees of the composite fields $\Bbb{Q}(E[m])\Bbb{Q}(\zeta_q)$ do not behave uniformly but depend on the common prime factors of $m$, $q$, and $M_E$.
\end{remark}

It can be shown that if $E$ is with CM by $\mathcal{O}_K$, then $\Bbb{Q}(E[m])/\Bbb{Q}$ is a meta-abelian extension, and thus Artin's conjecture holds for $\Bbb{Q}(E[m])/\Bbb{Q}$.\footnote{By a meta-abelian extension $L/K$, we mean that $L/K$ is a Galois extension with Galois group $G$, and $G$ admits an abelian normal subgroup $N$ such that $G/N$ is also abelian. For such an instance, all the irreducible representations of $G$ are monomial (namely, induced from 1-dimensional representations of subgroups of $G$), and thus Artin's conjecture is known. In fact, it can be further shown that Langlands reciprocity holds for such $L/K$ (see, e.g., \cite[Theorem 2.5]{Wong17}).} In what follows, we shall extend this result to $\Bbb{Q}(E[m])\Bbb{Q}(\zeta_q)/\Bbb{Q}$.

\begin{lemma}\label{AC}
In the notation as above, if $E$ is with CM by $\mathcal{O}_K$, then Artin's conjecture holds  for $\Bbb{Q}(E[m])\Bbb{Q}(\zeta_q)/\Bbb{Q}$. In addition, the biggest character degree $b( \Gal(\Bbb{Q}(E[m])\Bbb{Q}(\zeta_q)/\Bbb{Q} ))$ of the Galois group of $\Bbb{Q}(E[m])\Bbb{Q}(\zeta_q)/\Bbb{Q}$ is at most 2.
\end{lemma}

\begin{proof}
In \cite[Lemma 4]{MR83}, M.R. Murty  proved that there exists an ideal $\mathfrak{f}=\mathfrak{f}_E$ of $K$ such that
$$
K_{\m} \subseteq K(E[m]) \subseteq  K_{\mathfrak{fm}},
$$
where  $K_{\m}$ and  $K_{\mathfrak{fm}}$ are ray class fields of $K$ of levels $\m=m\mathcal{O}_K$ and $\mathfrak{fm}$, respectively. In addition, he showed that if $m \ge 3$, then $\Bbb{Q}(E[m]) = K(E[m])$. Consequently,  for $m \ge 3$,  $\Gal(\Bbb{Q}(E[m])/K )$ is abelian, and so there is an abelian normal subgroup $N$ of $G=\Gal(\Bbb{Q}(E[m])/\Bbb{Q} )$ such that the fixed field of $N$ in $\Bbb{Q}(E[m])/\Bbb{Q}$ is $K$. As $[K:\Bbb{Q}]=2$, $G /N $ is of order two and thus a nilpotent group. Therefore, by [12, Proposition 2.7], when $m\ge 3$, for any $\chi\in\Irr(\Gal(\Bbb{Q}(E[m])/\Bbb{Q} ))$, we have $\chi(1)\le 2$. On the other hand, since $\Gal(\Bbb{Q}(E[2])/\Bbb{Q})$ is a subgroup of $\GL_2(\Bbb{Z}/2\Bbb{Z})$, $|\Gal(\Bbb{Q}(E[2])/\Bbb{Q})|$ is either 1, 2, 3, or 6, which means that $\Gal(\Bbb{Q}(E[2])/\Bbb{Q})$ is an abelian group or $S_3$. Thus, $b(\Gal(\Bbb{Q}(E[2])/\Bbb{Q}) )\le 2$.

Now, since  $\Gal(\Bbb{Q}(\zeta_q)/\Bbb{Q})$ is abelian, it then follows that for $m\ge 2$, every irreducible character of $\Gal(\Bbb{Q}(E[m])/\Bbb{Q})\times  \Gal(\Bbb{Q}(\zeta_q)/\Bbb{Q})$ is of degree at most 2. Furthermore, recalling that the map $\sigma  \mapsto   (\sigma\mid_{\Bbb{Q}(E[m])}, \sigma\mid_{\Bbb{Q}(\zeta_q)})$ defines an injective homomorphism from $\Gal(\Bbb{Q}(E[m])\Bbb{Q}(\zeta_q)/\Bbb{Q})$ to $\Gal(\Bbb{Q}(E[m])/\Bbb{Q})\times  \Gal(\Bbb{Q}(\zeta_q)/\Bbb{Q})$, from the above discussion, we conclude that
$$
b(\Gal(\Bbb{Q}(E[m])\Bbb{Q}(\zeta_q)/\Bbb{Q}))
\le b( \Gal(\Bbb{Q}(E[m])/\Bbb{Q})\times  \Gal(\Bbb{Q}(\zeta_q)/\Bbb{Q}) )
\le 2.
$$
In other words, the set
$$
\cd(\Gal(\Bbb{Q}(E[m])\Bbb{Q}(\zeta_q)/\Bbb{Q})) :=\{\chi(1)\mid \chi\in \Irr(\Gal(\Bbb{Q}(E[m])\Bbb{Q}(\zeta_q)/\Bbb{Q}))\}
$$
is either $\{1 \}$ or $\{1,2 \}$, where the former instance implies that $\Gal(\Bbb{Q}(E[m])\Bbb{Q}(\zeta_q)/\Bbb{Q})$ is abelian. Hence, applying Artin reciprocity and \cite[Corollary 5.2]{Wong17}, Artin's conjecture holds for $\Bbb{Q}(E[m])\Bbb{Q}(\zeta_q)/\Bbb{Q}$.
\end{proof}

Now, we shall consider $L/K =\Bbb{Q}(E[m])\Bbb{Q}(\zeta_q)/\Bbb{Q}$. By \eqref{def-M} and the fact that $[\Bbb{Q}(E[m]):\Bbb{Q}]\ll m^2$, we deduce
$$
\log \mathcal{M}(\Bbb{Q}(E[m])\Bbb{Q}(\zeta_q)/\Bbb{Q}) \ll 
 \sum_{p\mid m qN_E} \log p  + \log (m^2q)
 \ll \log (m qN_E).
$$
Hence, when $(\log (mqN_E))^2 \ll \log x$, as $(\log \mathcal{A})^2 \ll \log x$ in this case, \eqref{VKMurty-est} implies that
\begin{align}\label{CDT}
 \begin{split}
\pi_C(x) &=\frac{|C|}{|G|} \Li(x) - \frac{|C|}{|G|}\chi_1(C) \Li(x^{\beta_1})+O\left(|C|^{\frac{1}{2}}  x
\exp (  -c_2' \sqrt{\log x} ) \right),
  \end{split}
\end{align}
for some $c_2'\in (0,c_2)$ with $c_2$ given in Theorem \ref{VKMurtythm}, where $\beta_1$ is the possible exceptional zero of the Artin $L$-function $L(s,\chi_1, \Bbb{Q}(E[m])\Bbb{Q}(\zeta_q)/\Bbb{Q})$. As $\chi_1$ is abelian, Artin reciprocity tells us that $\chi_1$ can be regarded as a Dirichlet character,\footnote{By a slight abuse of notation, we shall denote such a Dirichlet character by $\chi_1$.} and $L(s,\chi_1, \Bbb{Q}(E[m])\Bbb{Q}(\zeta_q)/\Bbb{Q})$ corresponds to the Dirichlet $L$-function attached to $\chi_1$. Moreover, by \eqref{logQ-bd}, the modulus $Q_1$ of $\chi_1$ satisfies 
$$
\log Q_1 \ll   \sum_{p\mid m qN_E} \log p  + \log (m^2q) \ll\log (m qN_E)
$$
which means that
\begin{equation}\label{E-logQ-bd}
\log Q_1 \le\kappa \log (m qN_E)
\end{equation}
for some absolute $\kappa>0$. Thus, by Theorem \ref{Siegel}, for any $\varepsilon>0$, we have
$$
\Li(x^{\beta_1})\ll \frac{x^{\beta_1}}{ \log x}
\ll \frac{x}{ \log x}  \exp \left( \log x \frac{-c_3(\varepsilon)}{Q_1^\varepsilon}   \right)
\le \frac{x}{ \log x}  \exp \left(\log x  \frac{ -c_3(\varepsilon)}{ e^{\varepsilon\kappa \log (mqN_E)}}   \right).
$$

Assume that $m \le (\log x)^A$ and $qN_E\le (\log x)^B$. Note that under this assumption, we have
$$
\log (mqN_E) \le (A+B)\log\log x,
$$
which particularly gives
$$
(\log mqN_E)^2 \ll_{A,B} \log x.
$$
Hence, choosing $\varepsilon=1/(2\kappa (A+ B))$, we arrive at
$$
\pi_C(x) =\frac{|C|}{|G|} \Li(x) +O\left( \frac{x}{ \log x}  \exp ( -c_3 \sqrt{\log x}) 
+|C|^{\frac{1}{2}}  x  \exp (  -c_2' \sqrt{\log x} ) \right)
$$
with $c_3 =c_3(1/(2\kappa (A+ B)))$. Thus, for any $A,B>0$, we have
\begin{equation}\label{CDT-3}
\pi_{E,m}(x,q,a) =\frac{ \gamma_{E,m}(q,a)}{[\Bbb{Q}(E[m])\Bbb{Q}(\zeta_q):\Bbb{Q}]} \Li(x)
+O\left(   x\exp (- c_{5, A, B}  \sqrt{\log x} ) \right),
\end{equation}
for some positive constant $c_{5, A, B}$  depending only on $A, B$, provided $m \le (\log x)^A$ and $qN_E\le (\log x)^B$. Here, the factor $\gamma_{E,m}(q,a)$ is defined as in Theorem \ref{AG-main-thm} (for the rationale of the appearance of $\gamma_{E,m}(q,a)$, see \cite[Sec. 3.2.1]{AG}).

\begin{remark} 
(i) The effective version of the Chebotarev density theorem due to Lagarias and  Odlyzko \cite{LO77} gives
\begin{equation}\label{LO-CDT}
\pi_C(x) =\frac{|C|}{|G|} \Li(x) + O\left(|C|   x  \exp (  -c_5' \sqrt{ (\log x)/n_L }  ) \right)
\end{equation}
with some absolute $c_5'>0$. (See also \cite[Lemma 2]{RM84}.\footnote{It is worth noting that as Artin's conjecture is known for abelian extensions, one may improve \cite[Theorem 1]{RM84} on an analogue of Artin's primitive root conjecture for abelian extensions by utilising \eqref{VKMurty-est} instead of \cite[Lemma 2]{RM84}.}) So, without Artin's conjecture, to have an estimate of similar strength as \eqref{CDT} (and its consequence \eqref{CDT-3}) for $L=\Bbb{Q}(E[m])\Bbb{Q}(\zeta_q)$, one would have to work over a much more restricted range of $m$ (which is roughly at most $\ll \sqrt{ \log x}$ since $\phi(m) \ll [\Bbb{Q}(E[m]):\Bbb{Q}] \ll m^2$). However, to prove \eqref{AMe-generalisation} and \eqref{SW-pie} for any given $A>0$, it is crucial to the uniformity of the estimate \eqref{CDT-3} for $m \le (\log x)^A$.

\noindent (ii) The Siegel-Walfisz type estimate \eqref{CDT-3} can be improved significantly under GRH. Indeed, assuming GRH, by the work of  Lagarias-Odlyzko  \cite{LO77}, for any elliptic curve $E/\Bbb{Q}$ (not necessarily with CM), one has 
$$
\pi_{E,m}(x,q,a) =\frac{ \gamma_{E,m}(q,a)}{[\Bbb{Q}(E[m])\Bbb{Q}(\zeta_q):\Bbb{Q}]} \Li(x)
 +O\left(  x^{\frac{1}{2}}  \log  ([\Bbb{Q}(E[m])\Bbb{Q}(\zeta_q):\Bbb{Q}]  mqN_E)\right),
$$ 
(see also \cite[Theorem 3.1 and Lemma 3.4]{CM04}). If $E$ has CM,  $[\Bbb{Q}(E[m])\Bbb{Q}(\zeta_q):\Bbb{Q}]\ll m^2 q$; otherwise, 
$[\Bbb{Q}(E[m])\Bbb{Q}(\zeta_q):\Bbb{Q}]\ll m^4 q$. Therefore, the above estimate becomes
\begin{equation}\label{GRH-CDT}
\pi_{E,m}(x,q,a) =\frac{ \gamma_{E,m}(q,a)}{[\Bbb{Q}(E[m])\Bbb{Q}(\zeta_q):\Bbb{Q}]} \Li(x)
 +O\left(  x^{\frac{1}{2}}  \log  ( mqN_E)\right),
\end{equation}
where the implied constant is absolute. As shall be seen in the proofs of Theorems \ref{main-thm-c-GRH} and \ref{main-thm-exp-cond}, \eqref{GRH-CDT} plays a crucial role in helping us to obtain the power-savings for the estimates involved.
 \end{remark} 

%
%
%
%
%

\section{Proof of Theorem \ref{main-thm-c}} \label{proof}

Let $E$ be an elliptic curve over $\Bbb{Q}$. Recall that the order of $\bar{E}(\Bbb{F}_p)$ can be written as $|\bar{E}(\Bbb{F}_p)|=p+1-a_p$ for some $a_p\in \Bbb{Z}$ satisfying Hasse's bound $|a_p|\le 2\sqrt{p}$. Also, by Lemma \ref{cri_d}, if $p\nmid N_E$ splits completely in $\Bbb{Q}(E[m])$, then $m^2$ divides $ |\bar{E}(\Bbb{F}_p)|$. Consequently, for such an instance, we must have $ m^2 \le p+1 +  2\sqrt{p} = (\sqrt{p} +1)^2    \le (\sqrt{x} +1)^2$ if $p\le x$.

In the remaining part of this section, we shall further assume that $E$ is with complex multiplication by the ring of integers $\mathcal{O}_K$ of an imaginary quadratic field $K$. Now, an application of the inclusion-exclusion principle,  Lemma \ref{cyc-cri}, and the above discussion give
\begin{align}\label{cut-pic}
 \begin{split}
\pi_c(x,E,q, a)&= \#\{p \le x \mid \text{$p\nmid N_E$, $p\equiv a\,\mymod{q}$, and $\bar{E}(\Bbb{F}_p)$ is cyclic}\}\\
&=\sum_{m\le y}\mu(m)\pi_{E,m}(x,q,a)  + O\left(   \frac{x}{y}\right),
  \end{split}
\end{align}
for any $3\le y\le \sqrt{x} +1$, where  $\pi_{E,m}(x,q,a)$ is defined as in Section \ref{pre}, and the big-O term follows from \eqref{tri_bd_Em} and the elementary bound
$$
\sum_{ m>y }  \frac{x}{m^2}\ll \frac{x}{y}.
$$

From \eqref{CDT-3}, it follows that the sum in \eqref{cut-pic} equals
\begin{align}\label{CDT-4}
 \begin{split}
\sum_{m\le y} \frac{ \gamma_{E,m}(q,a) \mu(m)}{[\Bbb{Q}(E[m])\Bbb{Q}(\zeta_q):\Bbb{Q}]} \Li(x)
 +O\left(   y x\exp (- c_{5, A, B}  \sqrt{\log x} ) \right),
  \end{split}
\end{align}
provided that $3\le y\le (\log x)^A$ and $qN_E \le (\log x)^B$, where $ \gamma_{E,m}(q,a)$ is defined as before. By the fact that $[\Bbb{Q}(E[m]):\Bbb{Q}]\gg \phi(m)^2$, it has been shown in the first displayed estimate of \cite[Sec. 4.1]{AG} that
$$
\sum_{m> y} \frac{\mu(m)^2}{[\Bbb{Q}(E[m])\Bbb{Q}(\zeta_q):\Bbb{Q}]} \ll \frac{1}{y},
$$
and thus
\begin{equation}\label{tail-cyclicity}
\sum_{m> y}\frac{ \gamma_{E,m}(q,a) \mu(m)}{[\Bbb{Q}(E[m])\Bbb{Q}(\zeta_q):\Bbb{Q}]} \Li(x)
\ll  \frac{x}{y\log x}.
\end{equation}
Finally, choosing $y=(\log x)^A$ and combining \eqref{cut-pic}, \eqref{CDT-4}, and \eqref{tail-cyclicity}, we establish
\begin{align*}
\pi_c(x,E,q, a)
&=\sum_{m\le (\log x)^A}\frac{ \gamma_{E,m}(q,a)\mu(m)}{[\Bbb{Q}(E[m])\Bbb{Q}(\zeta_q):\Bbb{Q}]} \Li(x)\\
&+ O\left(   x (\log x)^A\exp (- c_{5,A, B}  \sqrt{\log x} ) +   \frac{x}{(\log x)^A}\right)\\
&=\sum_{m=1}^{\infty}\frac{ \gamma_{E,m}(q,a)\mu(m)}{[\Bbb{Q}(E[m])\Bbb{Q}(\zeta_q):\Bbb{Q}]} \Li(x) + O_{A,B}\left(    \frac{x}{(\log x)^A}\right),
\end{align*}
whenever $qN_E \le (\log x)^B$, as desired.

\section{Proof of Theorem \ref{main-thm-c-GRH}} \label{proof-pic-GRH}

In this section, we will prove  Theorem \ref{main-thm-c-GRH}. As shall be seen, the key new input is the Brun-Titchmarsh theorem and its variant for number fields.

\subsection{A CM refinement}

We begin by recalling the work of Hinz and Lodemann on the Brun-Titchmarsh inequality for number fields. Given a number filed $F$, for any coprime integral ideals $\a,\q$ of $F$, we set
$$
\pi(x,\q,\a)=\#\{\p\subset \mathcal{O}_F  \mid \text{$\norm(\p)\le x$ and $\p\sim\a\,\mymod{\q}$}\}, 
$$
where $\p$ stands for a prime of $F$, and  $\p\sim\a\enspace\mymod{\q}$ means that $\p$ and $\a$ are in the same ray class of the ray class group modulo $\q$. By \cite[Theorem 4]{HL93} of Hinz and Lodemann, it is known that
for any $(\a,\q)=1$, if $\norm(\q) < x$, then
$$
\pi(x,\q,\a)\le \frac{2x}{h(\q)   \log ( x/ \norm(\q) )}\cdot \left( 1 +O\left( \frac{ \log \log ( 3x/ \norm(\q)) }{ \log ( x/ \norm(\q) )}  \right) \right) ,
$$
where $h(\q)$ denotes the cardinality of the ray class group modulo $\q$. 

Let $h_F$ denote the class number of $F$ and $r_1$ be the number of real embeddings of $F$. Recall that $h(\q)$ can be expressed as
$$
h(\q)=\frac{h_F2^{r_1}\Phi(\q)}{T(\q)},
$$
where $T(\q)$ is the number of residue classes $\mymod{\q}$ that contain a unit, and $\Phi(\q)$ is the number field analogue of Euler's totient function for $F$. Moreover, if $F$ is an imaginary quadratic field, then $T(\q)\le 6$. Therefore, if $F$ is an imaginary quadratic field of class number 1, one has
$$
\frac{1}{h(\q)} 
 \le \frac{6}{\Phi(\q)}.
$$

For an elliptic curve $E$ with CM by the full ring of integers $\mathcal{O}_K$ of an imaginary quadratic field $K$, $K$ must be of class number 1. As discussed in the proof of Lemma \ref{AC}, by the work of M.R. Murty \cite{MR83}, we know $K_{\m} \subseteq K(E[m])=  \Bbb{Q}(E[m])$  for $m\ge 3$, where $K_{\m}$ is the ray class field of $K$ of level $\m=m\mathcal{O}_K$. Hence, if a rational prime $p$ splits completely in $\Bbb{Q}(E[m])$, then for any $\p$ of $K$ that is above $p$, $\p$ must splits completely in  $K(E[m])$ and thus in $K_{\m}$. Consequently, for $m\ge 3$, we have 
$$
\pi_{E, m}(x) \le \pi(x,\m,{\bf 1})+\log N_E,
$$
 where ${\bf 1} = 1 \cdot \mathcal{O}_K = \mathcal{O}_K$, and $N_E$ is the conductor of $E$. Therefore, from the above discussion, it follows that for any fixed $\theta\in(0,\frac{1}{2})$, if $3\le y\le x^{\theta}$, then
\begin{align}\label{CM-middle-refined}
 \begin{split}
\sum_{y< m \le x^{\theta}} \pi_{E,m}(x,q,a)
&\ll  \sum_{y< m \le x^{\theta}} \frac{x}{\Phi(\m)   \log ( x/ \norm(\m) )}   + x^{\theta} \log N_E\\
&\ll_{\theta}  \sum_{y< m \le x^{\theta}} \frac{x}{\phi(m)^2   \log  x}  + x^{\theta} \log N_E .
 \end{split}
\end{align}
This bound, together with the estimate 
\begin{equation}\label{AG-lemma11}
\sum_{ m >X} \frac{1}{\phi(m)^2} \ll \frac{1}{X} 
\end{equation}
(see \cite[Lemma 11]{AG}) and  \eqref{tri_bd_Em}, then gives
\begin{equation}\label{middle-sum-piEm}
\sum_{y< m \le  \sqrt{x} +1 } \pi_{E,m}(x,q,a) \ll   \frac{x}{y \log  x}  + x^{\theta} \log N_E +  \frac{x}{x^{\theta}}.
\end{equation}
Thus, by an argument similar to the one leading to \eqref{cut-pic} and the above bound, we have
$$
\pi_c(x,E,q, a)
=\sum_{m\le y}\mu(m)\pi_{E,m}(x,q,a)  + O\left(  \frac{x}{y \log  x}  + x^{\theta} \log N_E +  \frac{x}{x^{\theta}}\right).
$$
Hence, using \eqref{GRH-CDT} and \eqref{tail-cyclicity}, we obtain 
\begin{align}\label{split-pic-CM-GRH}
 \begin{split}
\pi_{c}(x,E)
&= \sum_{m=1}^{\infty} \frac{ \gamma_{E,m}(q,a)}{[\Bbb{Q}(E[m])\Bbb{Q}(\zeta_q):\Bbb{Q}]}\Li(x)+ O\left( yx^{1/2}\log  ( qN_Ex)\right)\\
&+ O\left(   \frac{x}{y \log  x}  + x^{\theta} \log N_E +  x^{1-\theta}\right).
   \end{split}
\end{align}
Finally, choosing 
$$
y= \frac{x^{1/4}}{ (\log x)^{1/2} (\log ( qN_Ex))^{1/2}}
$$
and $\theta =\frac{1}{4}$ in \eqref{split-pic-CM-GRH}, we establish \eqref{pic-CM-GRH}.

\subsection{A non-CM refinement}\label{non-CM-refine}

Recall that in \cite[Sec. 3.2]{AG} (see, particularly, \cite[Eq. (15)-(18)]{AG}), it has been shown that under GRH, one has
\begin{equation}\label{pic-expression-GRH}
\pi_{c}(x,E,q,a) =  \mathfrak{c}_E(q,a) \Li(x) + O(yx^{1/2} \log (qN_E x) ) 
+ O( \Sigma'_2  + \Sigma_3 \Li(x) ),
\end{equation}
where
\begin{equation}\label{Sigma2}
\Sigma'_2 :=\sum_{y<m\le \sqrt{x}+1} \pi_{E,m}(x,q,a) \ll x^{1/2} \log x +\frac{x^{3/2}}{y^2q} H(q),
\end{equation}
and
\begin{equation}\label{Sigma3}
\Sigma_3 :=  \sum_{m> y}  \frac{\mu(m)^2}{[\Bbb{Q}(E[m])\Bbb{Q}(\zeta_q):\Bbb{Q}]} \ll  \frac{\tau_2(q_2)}{y^3  \phi(q) }M_E^3.
\end{equation}
In this section, we shall refine the estimates for $\Sigma_2$ (for ``small'' $q$) and $\Sigma_3$ (when $q$ and $M_E$ do not have many common prime factors) as follows.

First, by using Hasse's bound $|a_p|\le 2\sqrt{p} \le 2\sqrt{x}$ for $p\le x$, we have
\begin{align*}
 \pi_{E,m}(x,q,a)
&\le \#\{ p\le x \mid p\nmid 2N_E,  p\equiv a\,\mymod{q},  p\equiv 1\,\mymod{m}, m^2\mid  |\bar{E}(\Bbb{F}_p)| \}   +1\\
&\le \sum_{|b| \le 2\sqrt{x}} \#\{ p\le x \mid p\nmid 2N_E,  m\mid p-1, m^2\mid   p+1 -b, a_p=b \}  +1\\
&\le  \sum_{\substack{ |b| \le 2\sqrt{x},b\neq 2\\ m\mid b-2  }}  \sum_{\substack{ p\le x\\m^2\mid  p+1 -b }} 1  +\sum_{\substack{ p\le x\\ m^2\mid p-1  }} 1\\
&\ll  \sum_{\substack{ |b| \le 2\sqrt{x},b\neq 2\\ m\mid b-2   }} \frac{x}{\phi(m^2) \log (9x/m^2 )  }  +   \frac{x}{m^2},
\end{align*}
where the last estimate follows from the Burn-Titchmarsh theorem, provided that $m\le  \sqrt{x} +1$.\footnote{We shall note that this argument is inspired by the argument of Cojocaru and M.R. Murty \cite[Sec. 4]{CM04}, and our new input is the use of the Burn-Titchmarsh theorem in the last estimate.} Note that $t \log (9x/t^2 )$ is increasing for $0<t\le \sqrt{x} +1$, as its derivative is
$ \log (9x/t^2 )  +t( -2/t) \ge (\log 8) - 2 >0$, when $x$ is sufficiently large. Thus, we obtain
\begin{align*}
\sum_{y< m\le  \sqrt{x} +1} \pi_{E,m}(x,q,a)
& \ll \sum_{y< m\le  \sqrt{x} +1} \frac{\sqrt{x}}{m}  \frac{x}{m\phi(m) \log (9x/m^2 )  }  +\frac{x}{y}\\
& \le  \frac{x^{3/2}}{y\log (9x/y^2 )} \sum_{y< m\le  \sqrt{x} +1}   \frac{1}{m\phi(m) } +\frac{x}{y}.
\end{align*}
 By Abel's summation and the elementary estimate $\sum_{m\le t} \frac{1}{\phi(m)}  = \gamma \log t +O(1)$ for some constant $\gamma>0$,  the last sum above is $\ll \frac{1}{y}$ (see also \cite[Lemma 10]{AG}).
Hence, we derive
\begin{equation}\label{sigma2-primed-bd}
\Sigma'_2
 \ll \frac{x^{3/2}}{y^2 \log(3x/y^2)}  +\frac{x}{y}.
\end{equation}

To estimate the tail $\Sigma_3$, we shall closely follow the argument of \cite[p. 1294]{AG}. For each $q$, we set $q_1=\frac{q}{q_2}$ where $q_2$ denotes the largest divisor of $q$ that is coprime to $M_E$. As $\mu(m)^2=1$ if and only if $m$ is square-free, recalling that $M_E$ is square-free by its definition \eqref{def-ME}, we can write
\begin{equation*}
\Sigma_3= \sum_{\substack{ m> y\\ m \text{ square-free} }}  \frac{ 1}{[\Bbb{Q}(E[m])\Bbb{Q}(\zeta_q):\Bbb{Q}]}
=\sum_{\substack{ de  > y \\ d\mid M_E,  (e,M_E)=1 } }   \frac{ 1}{[\Bbb{Q}(E[de])\Bbb{Q}(\zeta_q):\Bbb{Q}]},
\end{equation*}
which, by the decomposition $q=q_1q_2$, is
\begin{align*}
&\sum_{d\mid M_E }  \frac{1}{[\Bbb{Q}(E[d])\Bbb{Q}(\zeta_{q_1}):\Bbb{Q}]  }
\sum_{\substack{ e  > y/d \\ (e,M_E)=1 } }   \frac{ 1}{[\Bbb{Q}(E[e])\Bbb{Q}(\zeta_{q_2}):\Bbb{Q}]}\\
&\le  \sum_{d\mid M_E }  \frac{1}{ \phi([d,q_1])  }
\sum_{\substack{ e  > y/d \\ (e,M_E)=1 } }   \frac{ [\Bbb{Q}(E[e])\cap \Bbb{Q}(\zeta_{q_2}):\Bbb{Q}]}{[\Bbb{Q}(E[e]):\Bbb{Q}][\Bbb{Q}(\zeta_{q_2}):\Bbb{Q}]}.
\end{align*}
(Here, we used the fact that the $d$-th cyclotomic field is contained in $\Bbb{Q}(E[d])$.) Moreover, it follows from the facts 
$$
[\Bbb{Q}(E[e]):\Bbb{Q}]\gg e^3 \phi(e) \text{ and } [\Bbb{Q}(E[e])\cap \Bbb{Q}(\zeta_{q_2}):\Bbb{Q}] =\phi((e,q_2))
$$ 
(see, e.g., \cite[p. 1294]{AG}) that
 the last sum above is
 $$
 \ll  \frac{1}{\phi(q_2)} \sum_{k\mid q_2} \phi(k) \sum_{\substack{ e  > y/d \\ (e,q_2)=k } } \frac{1}{e^3 \phi(e)}  \le
 \frac{1}{\phi(q_2)} \sum_{k\mid q_2} \phi(k) \sum_{ kr  > y/d  } \frac{1}{(kr)^3 \phi(kr)}.  
 $$
By the inequality $\phi(k)\phi(r) \le \phi (kr)$ and the estimate $ \sum_{ r  > X  } \frac{1}{r^3 \phi(r)} \ll X^{-3}$ (see, e.g., \cite[Lemma 10]{AG}), we see that the last quantity is
$$
 \le  \frac{1}{\phi(q_2)} \sum_{k\mid q_2}\frac{1}{k^3} \sum_{ r  > y/(dk)  } \frac{1}{r^3 \phi(r)}
 \ll  \frac{1}{\phi(q_2)} \sum_{k\mid q_2} \frac{1}{k^3} \frac{(dk)^3}{y^3} .   
$$
Thus, recalling that  $\phi([d,q_1])\phi((d,q_1)) = \phi(d)\phi(q_1) $, we derive
\begin{equation}\label{signa3-primed-bd}
\Sigma_3 
\ll \frac{\tau_2(q_2)}{y^3}
\sum_{d\mid M_E }  \frac{1}{ \phi([d,q_1])  } \frac{1}{\phi(q_2)}  d^3
= \frac{\tau_2(q_2)}{y^3 \phi(q)} \sum_{d\mid M_E }  \frac{\phi((d,q_1)) d^3}{ \phi(d)  }
= \frac{\tau_2(q_2)}{y^3 \phi(q)} R_{E,q_1},
\end{equation}
where $ R_{E,q_1}$ is defined as in \eqref{def-REq}. Thus, inserting \eqref{signa3-primed-bd} into the argument of  \cite[Sec. 3.2]{AG} (instead of using the bound given in \eqref{Sigma3}) yields the last assertion of the theorem.

Furthermore, by \eqref{pic-expression-GRH}, \eqref{sigma2-primed-bd}, and \eqref{signa3-primed-bd}
we can choose
\begin{equation}\label{choice-y-GRH}
y = \frac{x^{1/3}}{(\log x)^{1/3}(\log (qN_Ex))^{1/3} }
\end{equation}
to deduce \eqref{pic-aq-GRH}, which  completes the proof.

\section{Proof of Theorem \ref{main-thm-exp}} \label{proof-2}

Throughout this section, $p$ will denote a (rational) prime coprime to $N_E$.
We start by observing $d_pe_p =  |\bar{E}(\Bbb{F}_p)|=p+1-a_p$ and writing
$$
\pi_{e}(x,E,q,a)
=\sum_{\substack{p\le x\\ p\equiv a\mymod{q} }}e_p=\sum_{\substack{p\le x\\ p\equiv a\mymod{q} }}\frac{p}{d_p}+\sum_{\substack{p\le x\\ p\equiv a\mymod{q} }}  \frac{1}{d_p} (1 - a_p),
$$
where by Hasse's bound, the last sum is
$$
\ll  \sum_{\substack{p\le x\\ p\equiv a\mymod{q} }}  (1 + |a_p| )
 \ll \sum_{\substack{p\le x\\ p\equiv a\mymod{q} }} \sqrt{p} \ll \frac{x^{3/2}}{q}.
$$

For the main term, as done in \cite{FK14} and \cite{Wu14}, it follows from the identity
\begin{equation}\label{1/m}
\frac{1}{m}=\sum_{de\mid m}\frac{\mu(d)}{e}
\end{equation}
(see, e.g., the formula below \cite[Eq. (3.2)]{Wu14})
that
$$
\sum_{\substack{p\le x\\ p\equiv a\mymod{q} }}\frac{p}{d_p}
=\sum_{\substack{p\le x\\ p\equiv a\mymod{q} }}p \sum_{de\mid d_p}\frac{\mu(d)}{e}
=\sum_{m\le \sqrt{x} +1} \sum_{de\mid m}\frac{\mu(d)}{e} \sum_{\substack{p\le x \\ p\equiv a\mymod{q}\\m\mid d_p}}p. 
$$
Therefore, we can consider the  splitting  
\begin{align}\label{split-pie}
 \begin{split}
&\pi_{e}(x,E,q,a)\\
&= \sum_{m\le y} \sum_{de\mid m}\frac{\mu(d)}{e} \sum_{\substack{p\le x\\ p\equiv a\mymod{q}\\m\mid d_p}}p
 +  \sum_{y< m\le \sqrt{x} +1} \sum_{de\mid m}\frac{\mu(d)}{e} \sum_{\substack{p\le x\\ p\equiv a\mymod{q}\\m\mid d_p}}p +O\left(\frac{x^{3/2}}{q}\right),
   \end{split}
\end{align}
where $y = y(x)\le \sqrt{x} +1$ is a parameter  to be chosen later.
Since
$$
\left|\sum_{de\mid m}\frac{\mu(d)}{e}\right| \le \frac{1}{m}  \le 1
$$
(cf. \cite[Eq. (3.6)]{Wu14}), 
Lemma \ref{cri_d}, together with \eqref{tri_bd_Em}, yields that the last triple sum in \eqref{split-pie} is
$$
\ll  \sum_{y\le m\le \sqrt{x} +1}x \pi_{E,m}(x,q,a)  \ll \sum_{ y \le m\le \sqrt{x} +1} \frac{x^2}{m^2}\ll \frac{ x^2}{y}. 
$$

Now, applying Abel's summation, we deduce
\begin{align*}
\sum_{\substack{p\le x\\ p\equiv a\mymod{q}\\  m\mid d_p}}p
=&x\pi_{E,m}(x,q,a)-\int_2^{x}\pi_{E,m}(t,q,a)dt
\\
&=x\frac{ \gamma_{E,m}(q,a)}{[\Bbb{Q}(E[m])\Bbb{Q}(\zeta_q):\Bbb{Q}]}\Li(x) +x\mathcal{E}_{E,m}(x,q,a)\\
&-  \int_2^{x}\frac{ \gamma_{E,m}(q,a)}{[\Bbb{Q}(E[m])\Bbb{Q}(\zeta_q):\Bbb{Q}]}\Li(t)dt -\int_2^{x}\mathcal{E}_{E,m}(t,q,a)dt,
\end{align*} 
where $\mathcal{E}_{E,m}(x,q,a)= \pi_{E,m}(x,q,a)-\frac{\gamma_{E,m}(q,a)}{[\Bbb{Q}(E[m])\Bbb{Q}(\zeta_q):\Bbb{Q}]}\Li(x)$. Hence, by the expression
$$
\Li(x^2)=x\Li(x) -\int_2^{x}\Li(t)dt +O(1),
$$
 we have
\begin{align}\label{mid4}
 \begin{split}
\sum_{\substack{p\le x\\ p\equiv a\mymod{q}\\  m\mid d_p}}p
=\frac{ \gamma_{E,m}(q,a)}{[\Bbb{Q}(E[m])\Bbb{Q}(\zeta_q):\Bbb{Q}]}\Li(x^2)+O\left( x \max_{ t\le x}|\mathcal{E}_{E,m}(t,q,a)| +1\right).
 \end{split}
\end{align}
Again, by \eqref{CDT-3}, we obtain
\begin{align}\label{error-pie-uncond}
\sum_{m\le (\log x)^A }  \max_{ t\le x}|\mathcal{E}_{E,m}(t,q,a) |\ll  x (\log x)^A\exp (- c_{5,A, B}  \sqrt{\log x} ).
\end{align}
Therefore, by \eqref{split-pie}, \eqref{mid4} and  \eqref{error-pie-uncond}, choosing $y = (\log x)^A $, we conclude that
\begin{align*}
\pi_{e}(x,E)&=\sum_{1\le m\le (\log x)^A }  \sum_{de\mid m}\frac{\mu(d)}{e} \frac{ \gamma_{E,m}(q,a)}{[\Bbb{Q}(E[m])\Bbb{Q}(\zeta_q):\Bbb{Q}]}\Li(x^2)\\
&+O\left(   x^2 (\log x)^A\exp (- c_{5,A, B}  \sqrt{\log x} ) + \frac{x^2}{(\log x)^A} +\frac{x^{3/2}}{q}\right)  .
\end{align*}
Note that
\begin{equation}\label{tail-2}
\sum_{m> y} \sum_{de\mid m}\frac{\mu(d)}{e}  \frac{ \gamma_{E,m}(q,a) }{[\Bbb{Q}(E[m])\Bbb{Q}(\zeta_q):\Bbb{Q}]} \Li(x^2)
\ll \sum_{m> y} \frac{1}{\phi(m)^2} \frac{x^2}{\log x}
\ll \frac{1}{y} \frac{x^2}{\log x},
\end{equation}
where the last estimate follows from \eqref{AG-lemma11}. 
Hence, we finally arrive at
\begin{align*}
  \begin{split}
\pi_{e}(x,E,q,a)
= \e_{E,q,a}\Li(x^2)+O_{A,B}\left(   \frac{x^2}{(\log x)^A} \right)
 \end{split}
\end{align*}
whenever $qN_E\le (\log x)^B$.

\section{Conditional estimates for $\pi_{e}(x,E,q,a)$} \label{proof-pie-GRH}

Throughout this section, we shall assume GRH and apply \eqref{GRH-CDT}. More precisely, by \eqref{GRH-CDT} and \eqref{mid4}, under GRH, we have
\begin{align}\label{mid-0}
 \begin{split}
\sum_{\substack{p\le x\\ p\equiv a\mymod{q}\\  m\mid d_p}}p
=\frac{ \gamma_{E,m}(q,a) }{[\Bbb{Q}(E[m])\Bbb{Q}(\zeta_q):\Bbb{Q}]}\Li(x^2)+O\left( x^{3/2}\log  ( mqN_E)\right).
 \end{split}
\end{align}
Also, bounding each $p$ trivially by $x$ and using Lemma \ref{cri_d}, we have
\begin{equation}\label{pie-middle}
\sum_{y\le m\le  \sqrt{x} +1} \sum_{de\mid m}\frac{\mu(d)}{e} \sum_{\substack{p\le x\\ p\equiv a\mymod{q}\\m\mid d_p}}p 
\ll x \sum_{y< m \le \sqrt{x} +1} \pi_{E,m}(x,q,a).
\end{equation}

\subsection{Elliptic curves with CM}

We begin by noting that if $E$ has CM, \eqref{sum-y-piEm} and \eqref{middle-sum-piEm} tell us that the right of \eqref{pie-middle} is $\ll   x\mathcal{E}_1 (x) $ with
\begin{align*}
\mathcal{E}_1 (x)  :=     \min\left\{   \left( \frac{\sqrt{x}}{q} + \frac{\sqrt{x}\log x}{q^2} + \frac{x}{yq^3} \right)G_D(a,q)   ,   \frac{x}{y \log  x}  + x^{\theta} \log N_E +  \frac{x}{x^{\theta}}\right\}.
\end{align*}
Putting \eqref{split-pie}, \eqref{tail-2}, \eqref{mid-0}, and this bound together  then yields 
\begin{align}\label{split-pie-CM}
 \begin{split}
\pi_{e}(x,E,q,a)
&= \sum_{m=1}^{\infty} \sum_{de\mid m}\frac{\mu(d)}{e} \frac{ \gamma_{E,m}(q,a)}{[\Bbb{Q}(E[m])\Bbb{Q}(\zeta_q):\Bbb{Q}]}\Li(x^2)+ O\left( yx^{3/2}\log  ( qN_Ex)\right)\\
&+ O\left( x \mathcal{E}_1 (x)   + \frac{x^2}{y\log x} +\frac{x^{3/2}}{ q}\right).
   \end{split}
\end{align}
Finally, as done in \cite[p. 1301]{AG}, applying \cite[Lemma 2.4]{GK91} to find $y \in[2 q ,
x^{1/2}]$ to balance the error terms in \eqref{split-pie-CM},  we deduce \eqref{pie-CM-GRH}.

%
%

\subsection{Non-CM elliptic curves}
Assume that $E$ is non-CM. As discussed in Section \ref{non-CM-refine}, by \eqref{sigma2-primed-bd} and \eqref{Sigma2}, the last sum in \eqref{pie-middle} is
$$
\ll
\mathcal{E}_2(x):= \min\left\{ x^{1/2} \log x +\frac{x^{3/2}}{y^2q} H(q), \frac{x^{3/2}}{y^2 \log(3x/y^2)}  +\frac{x}{y} \right\}
$$
provided that $ 2q\le y\le \sqrt{x}$, and thus the triple sum on the left of \eqref{pie-middle} is $\ll x\mathcal{E}_2(x)  $.

By \eqref{1/m},  we have
\begin{equation}\label{sigma3''}
\sum_{m> y} \sum_{de\mid m}\frac{\mu(d)}{e} \frac{\gamma_{E,m}(q,a) }{[\Bbb{Q}(E[m])\Bbb{Q}(\zeta_q):\Bbb{Q}]}
\ll \sum_{m> y} \frac{1}{m} \frac{1}{[\Bbb{Q}(E[m])\Bbb{Q}(\zeta_q):\Bbb{Q}]}.
\end{equation}
From an argument  analogous to the one leading to \eqref{signa3-primed-bd}, it follows that the right of \eqref{sigma3''} is
$$
\ll
\sum_{d\mid M_E^\infty }  \frac{1}{d [\Bbb{Q}(E[d])\Bbb{Q}(\zeta_{q_1}):\Bbb{Q}]  }\frac{1}{\phi(q_2)} \sum_{k\mid q_2} \frac{1}{k^3} \frac{(dk)^3}{y^3} .
$$
Unfortunately, if we used the bound $[\Bbb{Q}(E[d])\Bbb{Q}(\zeta_{q_1}):\Bbb{Q}]\ge \phi([d,q_1]) $  as before, the above sum over $d$ would not converge. To resolve this issue, we recall that by Serre's open image theorem, Freiberg and  Kurlberg \cite[Proposition  3.2]{FK14} showed that there exists a constant $B_E$, depending only on $E$, such that $ B_E \cdot [\Bbb{Q}(E[m]):\Bbb{Q}]\ge |\GL_2(\Bbb{Z}/m\Bbb{Z})| \gg m^3 \phi(m)$, for any $m\in\Bbb{N}$, whenever $E$ is non-CM. From which, we derive the upper bound
$$
\ll
\sum_{d\mid M_E^\infty }  \frac{B_E}{d^4 \phi(d) }\frac{1}{\phi(q_2)} \sum_{k\mid q_2} \frac{1}{k^3} \frac{(dk)^3}{y^3}    =\frac{\tau_2(q_2)}{y^3 \phi(q_2)}\sum_{d\mid M_E^\infty }  \frac{B_E}{d \phi(d) }.
$$

Thus,  by \eqref{split-pie}, \eqref{mid-0}, and the above discussion, we obtain
\begin{align*} 
 \begin{split}
\pi_{e}(x,E,q,a)
&= \sum_{m=1}^{\infty} \sum_{de\mid m}\frac{\mu(d)}{e} \frac{ \gamma_{E,m}(q,a)}{[\Bbb{Q}(E[m])\Bbb{Q}(\zeta_q):\Bbb{Q}]}\Li(x^2)
+ O\left( yx^{3/2}\log  ( qN_E x)\right) \\
&+     O\left(x \mathcal{E}_2 (x) +\frac{\tau_2(q_2) }{ \phi(q_2) } \sum_{d\mid M_E^\infty }  \frac{B_E}{d \phi(d) } \frac{x^2}{y^3\log x} +\frac{x^{3/2}}{q} \right) .
   \end{split}
\end{align*}
Hence, balancing the errors as in \cite[Sec. 3.2.3]{AG} gives  the desired estimate \eqref{pie-nonCM-GRH-2}. 

Finally, to derive \eqref{pie-nonCM-GRH-1}, we instead use the bound $[\Bbb{Q}(E[m])\Bbb{Q}(\zeta_q):\Bbb{Q}]\ge \phi(m)$ in \eqref{sigma3''} so that the resulting error is $\ll y^{-1}$. Consequently, the last big-O term above can be replaced by $O(x \mathcal{E}_2 (x) + \frac{x^2}{y\log x} +\frac{x^{3/2}}{q}  ) $, and choosing $y$ as in \eqref{choice-y-GRH} yields \eqref{pie-nonCM-GRH-1}. 



\section{When the Landau-Siegel zero is not too close to 1}\label{final-rmk}

In this section, we shall assume \eqref{lower-bd-L}. Suppose that there is an exceptional Dirichlet character $\chi_1$ modulo $Q_1$ such that $L(s,\chi_1)$ admits a Landau-Siegel zero $\beta_1$. Arguing classically, by the mean value theorem, we have
$$
1- \beta_1 = \frac{ L(1,\chi_1)}{L'(\sigma_1,\chi_1)} 
$$
for some $\sigma_1\in (\beta_1, 1)$. This, combined with \eqref{lower-bd-L} and the well-known estimate $L'(\sigma_1,\chi_1) =O( (\log Q_1)^2)$,  yields
$$
1- \beta_1>\frac{c_6}{(\log Q_1)^{S+2}}
$$ 
for some $c_6>0$. From this lower bound, one has
\begin{equation}\label{x-beta1-bd}
x^{\beta_1} \ll x \exp \left( -c_6(\log x) (\log Q_1)^{-S-2}   \right)  \ll x \exp ( -c_6 \sqrt{\log x}  )
\end{equation}
whenever $(\log Q_1)^{S+2}  \le \sqrt{\log x}$ or, equivalently, $Q_1\le \exp \left( (\log x)^{1/(2S+4)}   \right) $, which leads to the following improvement of the Siegel-Walfisz theorem:
$$
\pi(x,q,a) =\frac{1}{\phi(q)} \Li(x) + O\left(x \exp ( -c_6' \sqrt{\log x} )  \right)
$$
uniformly in  $q\le \exp \left( (\log x)^{1/(2S+4)}   \right) $ for some abosulte $c_6'\in(0,c_6)$.  (Cf. \cite[\S\S 21-22]{Da3rd}.)

Now, in the same notation of Section \ref{pre}, for any CM elliptic curve $E$ of conductor $N_E$, assuming that $\log m    \le  \frac{1}{2\kappa} (\log x)^{1/(2S+4)}$
and
$ \log (qN_E )  \le \frac{1}{2\kappa} (\log x)^{1/(2S+4)}$, by \eqref{E-logQ-bd}, we know that
$$
\log Q_1  \le\kappa \log (m qN_E)   \le (\log x)^{1/(2S+4)}.
$$
Therefore, by \eqref{CDT} and \eqref{x-beta1-bd}, we arrive at
\begin{align*}
 \begin{split}
\pi_C(x) &=\frac{|C|}{|G|} \Li(x) + O\left(  |C|^{\frac{1}{2}}  x 
\exp (  -c'_2 \sqrt{\log x} )   + x \exp ( -c_6\sqrt{\log x})\right).
  \end{split}
\end{align*}
Hence, under the assumption of  \eqref{lower-bd-L}, the estimate  \eqref{CDT-3} can be improved as
\begin{equation}\label{Z-bd-CDT}
\pi_{E,m}(x,q,a) =\frac{ \gamma_{E,m}(q,a)}{[\Bbb{Q}(E[m])\Bbb{Q}(\zeta_q):\Bbb{Q}]} \Li(x)
+O\left(   x\exp (- c_7  \sqrt{\log x} ) \right),
\end{equation}
uniformly in $\log m    \le  \frac{1}{2\kappa} (\log x)^{1/(2S+4)}$
and
$ \log (qN_E )  \le \frac{1}{2\kappa} (\log x)^{1/(2S+4)}$, for some $c_7>0$.

Thus, gathering \eqref{cut-pic},  \eqref{tail-cyclicity}, and \eqref{Z-bd-CDT} (instead of \eqref{CDT-4}), 
we get
\begin{align*}
\pi_c(x,E,q, a)
=\sum_{m=1}^\infty \frac{ \gamma_{E,m}(q,a)\mu(m)}{[\Bbb{Q}(E[m])\Bbb{Q}(\zeta_q):\Bbb{Q}]} \Li(x)+ O\left(   y x\exp (- c_7  \sqrt{\log x} ) +   \frac{x}{y}\right)
\end{align*}
whenever $ \log (qN_E )  \le \frac{1}{2\kappa} (\log x)^{1/(2S+4)}$. This, together with the choice
$$
y= \exp \left(  \min\left\{ \frac{1}{2\kappa}, \frac{c_7}{2} \right\} (\log x)^{1/(2S+4)}   \right),
$$
then yields the first claimed estimate of Theorem \ref{main-thm-NLS}. Finally, we conclude this section by noting that the last estimate of Theorem \ref{main-thm-NLS} follows similarly while using \eqref{Z-bd-CDT} (instead of \eqref{CDT-4}) in the argument starting from \eqref{mid4}.

\section*{Acknowledgments}
The author thanks Professors Amir Akbary,  Wen-Ching Winnie Li, and Ram Murty for their encouragement as well as helpful comments and suggestions. He is also grateful to the referee for the careful reading and insightful comments.


\end{document}